\documentclass[reqno]{amsart}
\usepackage{graphicx}
\usepackage{amssymb}
\usepackage{amsfonts}
\usepackage{indentfirst}
\usepackage{amssymb,amsmath,amsthm}
\usepackage{latexsym,bm}
\usepackage{graphicx}
\usepackage{times}
\usepackage{amsfonts}
\usepackage{indentfirst}
\usepackage[colorlinks,linkcolor=black,anchorcolor=blue,citecolor=blue]{hyperref}
\usepackage[dvipsnames]{pstricks}
\newtheorem{theorem}{Theorem}
\newtheorem{proposition}{Proposition}

\newtheorem{lemma}{Lemma}
\numberwithin{equation}{section}

\def\Z{\mathbb{Z}}

\def\Zp{\mathbb{Z}_{p}}
\def\Qp{\mathbb{Q}_{p}}

\def\Q{\mathbb{Q}}

\def\P{\mathbb{P}^1(\mathbb{Q}_{p})}

\def\Z2{\mathbb{Z}_{2}}
\def\m2#1{\ ({\rm mod} \ 2^{#1})}

\begin{document}

\title{Rational map $ax+1/x$ on  the projective line over $\mathbb{Q}_{p}$ }

\author{Shilei Fan}

\address{School of Mathematics and Statistics, Hubei Key Laboratory of Mathematical Sciences, Central  China Normal University,  Wuhan, 430079, China  \&\& Aix-Marseille Universit\'e, Centrale Marseille, CNRS, Institut de Math\'ematiques de Marseille, UMR7373, 39 Rue F. Joliot Curie 13453, Marseille, FRANCE}
\email{slfan@mail.ccnu.edu.cn}

\author{Lingmin Liao}
\address{LAMA, UMR 8050, CNRS,
Universit\'e Paris-Est Cr\'eteil Val de Marne, 61 Avenue du
G\'en\'eral de Gaulle, 94010 Cr\'eteil Cedex, France}
\email{lingmin.liao@u-pec.fr}

\thanks{The research of S. L. FAN on this project  has received funding from the European  Research Council(ERC) under the European Union's Horizon 2020 research  and innovation programme under grant agreement No 647133(ICHAOS). It has also been  supported by NSF of China (Grant No.s 11401236 and 11471132).}
\begin{abstract}
 The dynamical structure of the rational map $ax+1/x$ on the projective line $\P$ over the field $\mathbb{Q}_p$ of $p$-adic numbers 
%is studied in this note.
%For each Chebyshev polynomial, the dynamical structure
 is described for $p\geq 3$. %by showing all its minimal subsystems and their attracting basins.

\end{abstract}
%\thanks{S.F. was partially supported by NSF of China (Grant No. 11231009) and CNRS program (PICS No.5727), L.L. was partially supported by 12R03191A - MUTADIS (France) and the project PHC Orchid of MAE and MESR of France}
 %{\it Minist\`eres des Affaires \'Etraaang\`eres} and {\it Minist\`eres de l'Enseignement Sup\'erieur et de la Recherche}
\maketitle

%% main text

\section{Introduction}

For a prime number $p$,  let $\mathbb{Q}_p$ be the field of $p$-adic numbers and $\P$ its projective line. 
Recently, polynomials and rational maps of $\mathbb{Q}_p$ have been studied as dynamical systems on $\mathbb{Q}_p$ or $\P$. It turns out that these $p$-adic dynamical systems are quite different to the dynamical systems in Euclidean spaces. See for example, \cite{Anashin-Khrennikov-AAD,Baker-Rumely-book,SilvermanGTM241} and their bibliographies therein.  

For polynomials and rational maps of $\mathbb{Q}_p$, we can find two different kinds of subsystems exhibiting totally different dynamical behavior. One is $1$-Lipschitz dynamical systems and the other is $p$-adic repellers. 

A $1$-Lipschitz $p$-adic dynamical system can usually be fully described by showing all its minimal subsystems. Polynomials with coefficients in the ring $\Zp$ of $p$-adic integers and rational maps with good reduction are two important families of $1$-Lipschitz dynamical systems. In \cite{FL11}, the authors proved the following structure theorem for polynomials in $\mathbb{Z}_p[x]$. The same structure theorem for good reduction maps with degree at least $2$ was proved in \cite{FFLW2016}. 
%This describes to some extent the dynamical behavior of the system.
\begin{theorem}[\cite{FL11}, Theorem 1]\label{thm-decomposition}
 Let $f \in \mathbb{Z}_p[x]$ be a polynomial of integral coefficients with degree $\geq 2$. We have the following decomposition
$$
     \mathbb{Z}_p = \mathcal{P} \bigsqcup \mathcal{M} \bigsqcup \mathcal{B}
$$
where $ \mathcal{P}$ is the finite set consisting of all periodic points of
$f$, $\mathcal{M}= \bigsqcup_i \mathcal{M}_i$ is the union of all (at most countably
many) clopen invariant sets such that each $\mathcal{M}_i$ is a finite union
of balls and each subsystem $f: \mathcal{M}_i \to \mathcal{M}_i$ is minimal, and each
point in $\mathcal{B}$ lies in the attracting basin of a periodic orbit or of
a minimal subsystem.
\end{theorem}
The decomposition in Theorem \ref{thm-decomposition} is usually referred to as a \emph{minimal decomposition} and the invariant subsets $\mathcal{M}_i$ are called {\em minimal components}. 
In the literature, the minimality of the polynomial (or $1$-Lipschitz) dynamical systems on the whole space $\mathbb{Z}_p$ was widely studied \cite{Ana94, AKY11, CP11Ergodic, DP09, FLYZ07, FLSq, FLCh, Jeong2013,OZ75}.

\medskip
The $p$-adic repellers (see definition in page 220 of  \cite{FLWZ07}) are expanding dynamical systems which have positive topological entropy and thus exhibit chaotic behaviors. In \cite{FLWZ07}, it is proved that a transitive $p$-adic repeller is isometrically (hence topologically) conjugate to a subshift of finite type where a suitable metric is defined. A general method was also proposed in \cite{FLWZ07} to find subshifts of finite type subsystems in a $p$-adic polynomial dynamical system.  We remark that Thiran, Verstegen and Weyers \cite{TVW89} and Dremov, Shabat and Vytnova \cite{DSV06} studied the chaotic behavior of $p$-adic quadratic polynomial dynamical systems. Woodcock and Smart \cite{WS98} proved that the so-called $p$-adic
logistic map ${x^p-x \over p}$ is topologically conjugate to the full shift on the symbolic system with $p$ symbols.

On the other side, the dynamical properties of the fixed points of the rational maps have been studied in the space $\mathbb{C}_p$ of $p$-adic complex numbers \cite{ARS13,KM06,MR04,Satt2015} and in the adelic space \cite{DKM07}. The Fatou and Julia theory of the rational maps on $\mathbb{C}_p$, and  on the Berkovich space over $\mathbb{C}_p$, are also developed 
\cite{Benedetto-Hyperbolic-maps,Baker-Rumely-book,Hsia-periodic-points-closure,Rivera-Letelier-Dynamique-rationnelles-corps-locaux,SilvermanGTM241}.  
However,  the global dynamical structure of rational maps on $\Q_p$ remains unclear, though the rational maps of degree one are totally characterized in \cite{FFLW2014} .

%Let $\phi\in \Qp(z)$ be a rational map of degree  $d\geq 2$. Then $\phi$ induces a dynamical system on the projective line $\P$ over $\Q_p$, denoted  by $(\P, \phi)$. Let $E\subset \P$ be a subset such that $\phi(E)\subset E$. Then restricted to $E$, $\phi$ defines a subsystem $(E, \phi |_E)$. The subsystem $(E, \phi |_E)$ is called {\it minimal} if for any point $x\in E$, the orbit of $x$ under $\phi$ is dense in $E$.
% In this article, we suppose that $\phi$ has good reduction (see the definition below).  
% The minimality of $(\P, \phi)$ and its subsystems will be fully investigated.  As we will see, any rational map  $\phi$ having good reduction is $1$-Lipschitz continuous on $\P$ which is equipped with its spherical metric. This suggests that  $(\P, \phi)$
%shares many properties with the polynomial dynamics on $\Zp$,
%which are $1$-Lipschtz with respect to the metric induced by the
%$p$-adic absolute value.

%A dynamical system is a continuous transformation acting on a topological space. 
%To understand a dynamical system, we want to know how a point moves under the iteration of the transformation. A (sub)-system is called {\it minimal} if every orbit is dense in the (sub)-space. 

In the present article, we suppose $p\geq 3$ and investigate the following special class of rational maps of degree $2$:
 \begin{align}
\phi(x)=ax+\frac{1}{x}, \quad a\in \Q_p\setminus \{0\}.
\end{align}
We distinguish three cases: (1) $|a|_p=1$, (2) $|a|_p>1$, (3) $|a|_p<1$. 

Observe that  $$\phi(x)=ax+\frac{1}{x}=\frac{ax^2-1}{x}.$$ When $|a|_p=1$, the map $\phi$ has good reduction (see definition in page 58 of \cite{SilvermanGTM241}).
By Theorem 1.2 of \cite{FFLW2016}, we immediately have the following structure theorem.
\begin{theorem}\label{thm2}
Let $\phi(x)=ax+\frac{1}{x}$  with $a\in \Q_p\setminus \{0\}$.
 If $|a|_p=1$, the dynamical system $(\P,\phi)$ can be decomposed into 
$$\P= \mathcal{P} \bigsqcup \mathcal{M} \bigsqcup \mathcal{B}
$$
where $ \mathcal{P}$ is the  finite set  consisting of all periodic points of
$\phi$, $\mathcal{M}= \bigsqcup_i \mathcal{M}_i$ is the union of all (at most countably
many) clopen invariant sets such that each $\mathcal{M}_i$ is a finite union
of balls and each subsystem $\phi: \mathcal{M}_i \to \mathcal{M}_i$ is minimal, and each
point in $\mathcal{B}$ lies in the attracting basin of a periodic orbit or of
a minimal subsystem.
\end{theorem}

We are thus left to study the rest two cases. The following are our main theorems. We remark that in both cases, $\infty\in \P$ is a fixed point of $\phi$ with multiplier $1/a$.

\begin{theorem}\label{main-thm-1}
Let $\phi(x)=ax+\frac{1}{x}$  with $ a\in \Q_p\setminus \{0\}$.  If $|a|_p> 1$, the  dynamical structure of the system $(\P,\phi)$ is described  as follows.
\begin{itemize}
\item[(1)] If $\sqrt{1-a}\notin\Q_p$, then \[\forall x\in \Q_p, \quad \lim_{n\to\infty} \phi^n(x)=\infty.\]
\item[(2)] If $\sqrt{1-a}\in\Q_p$, then there exists an invariant set $\mathcal{J}$ such that the subsystem $(\mathcal{J},\phi)$ is topologically conjugate to $(\Sigma_{2}, \sigma),$ the full shift of two symbols. Further, 
%$$(\mathcal{J},\phi)\sim(\Sigma_{2}, \sigma),$$
 $$\forall x\in \Qp\setminus \mathcal{J}, \ \ \lim_{n\to \infty}\phi^n(x)=\infty.$$

\end{itemize} 
\end{theorem}

\begin{theorem}\label{main-thm-2}
Let $\phi(x)=ax+\frac{1}{x}$  with $ a\in \Q_p\setminus \{0\}$.  
 If $|a|_p<1$,   we distinguish two cases.
 \begin{itemize}
 \item[(1)] If $\sqrt{-a}\notin \Q_p$, then  $\phi(0)=\phi(\infty)=\infty$ is a repelling fixed point and the subsystem $(\Q_p\setminus\{0\}, \phi)$ has a minimal decomposition as stated in Theorem \ref{thm2}.
 \item[(2)] If $\sqrt{-a}\in \Q_p$, then  the dynamical system $(\P,\phi)$ has a subsystem which is conjugate to a subshift of finite type with positive entropy.
 \end{itemize}
 
\end{theorem}

\medskip
\section{preliminaries}
Let $p\geq 2$ be a prime number.
Any nonzero rational number $r\in \mathbb{Q}$ can be written as
$r =p^v \frac{a}{b}$ where $v, a, b\in \mathbb{Z}$ and $a$, $b$ are not divisible by $p$. %$(p, a)=1$ and $(p, b)=1$
%(here $(x, y)$ denotes the greatest common divisor of two integers $x$ and $y$). 
 Define $v_p(r)=v$ and
$|r|_p = p^{-v_p(r)}$ for $r\not=0$ and $|0|_p=0$.
Then $|\cdot|_p$ is a non-Archimedean absolute value on $\Q$. That means\\
\indent (i)  \ \ $|r|_p\ge 0$ with equality only for $r=0$; \\
\indent (ii) \ $|r s|_p=|r|_p |s|_p$;\\
\indent (iii) $|r+s|_p\le \max\{ |r|_p, |s|_p\}$.\\
The field $\mathbb{Q}_p$ of $p$-adic numbers is the completion of $\mathbb{Q}$ under the absolute value 
$|\cdot|_p$. 
Actually, any $x\in \Qp$ can be written as
$$
   x = \sum_{n=v_p(x)}^\infty a_n p^n    \quad (v_p(x) \in \mathbb{Z}, a_n \in\{0, 1, 2, \cdots, p-1\} \hbox{ and  }a_{v_p(x)}\neq 0).
$$
Here, the integer $v_p(x)$ is called the {\em $p$-valuation} of $x$.

 Any point in the projective line $\P$ of $\Q_p$ can be given in homogeneous coordinates by a pair
$[x_1 : x_2]$
of points in $\Qp$ which are not both zero. Two such pairs are equal if they differ by an overall (nonzero) factor $\lambda \in \Q_p^*$:
$$[x_1 : x_2] = [\lambda x_1 : \lambda x_2].$$
The field $\Qp$ is identified with the subset of $\P$ given by
$$\left\{[x : 1] \in \mathbb P^1(\Qp) \mid x \in \Qp\right\}.$$
This subset contains  all points  in $\P$ except one: the point  of infinity, which may be given as
$\infty = [1 : 0].$

The spherical metric defined on $\P$ is analogous to the standard spherical metric on the
Riemann sphere. If $P=[x_1,y_1]$ and $Q=[x_2,y_2]$ are two points in $\P$, we define
 $$\rho(P,Q)=\frac{|x_1y_2-x_2y_1|_p}{\max\{|x_1|_{p},|y_1|_{p}\}\max\{|x_2|_{p},|y_{2}|_{p}\}}$$
 or, viewing $\mathbb{P}^{1}(\Qp)$ as $\Qp\cup\{\infty\}$, for $z_1,z_2 \in \Qp\cup \{\infty\}$
 we define
 $$\rho(z_1,z_2)=\frac{|z_1-z_2|_{p}}{\max\{|z_1|_{p},1\}\max\{|z_2|_{p},1\}}  \qquad\mbox{if~}z_{1},z_{2}\in \Qp,$$
 and
 $$\rho(z,\infty)=\left\{
                    \begin{array}{ll}
                      1, & \mbox{if $|z|_{p}\leq 1$,} \\
                      1/|z|_{p}, & \mbox{if $|z|_{p}> 1$.}
                    \end{array}
                  \right.
 $$

Remark that the restriction of the spherical metric on the ring $\Zp:=\{x\in \Q_p, |x|\leq 1\}$ of $p$-adic integers is the same as the metric induced by the absolute value $|\cdot|_p$.

A rational map $\phi \in \Qp(z)$ induces a transformation  on  $\P$.
Rational maps  are  always Lipschitz continuous on $\P$
with respect to the spherical metric (see \cite[Theorem 2.14.]{SilvermanGTM241}).

In $\Qp$, we denote by $D(a,r):=\{x\in \Q_p: |x|_p\leq r\}$ the closed disk centered at $a$ with radius $r$ and by $S(a,r):=\{x\in \Q_p: |x|_p=r\}$ its corresponding sphere. A closed disk in $\P$ is either a closed disk in $\Qp$ or the complement of an open disk in $\Qp$. 

We recall some standard terminology of the theory of dynamical systems. If $\phi(x_0)= x_0$ then $x_0$ is called a fixed point of $\phi$. The set of all fixed points of $f$ is denoted by ${\rm Fix}(f).$  An important role in iteration theory is played by the periodic points.   By definition, $x_0$ is called a \emph{periodic point }of $\phi$ if $\phi^{n}(x_0)=x_0$ for some $n\geq 1$. In this case, $n$ is called a \emph{period} of $x_0$, and the smallest $n$ with this property is called the \emph{exact period} of $x_0$.

For a  periodic  point $x_0\in \Q_p$ of exact period $n$, $(\phi^{n})^{\prime}(x_0)$ is called the multiplier of $x_0$. Remark that the  multiplier is invariant by changing of coordinate.  If $\infty$ is a periodic point of period $n$, then the multiplier of $\infty$ is $\psi^{\prime}(0)$, where $\psi(x)=\frac{1}{\phi^n(1/x)}$.
A  periodic point is called \emph{attracting, indifferent,} or \emph{repelling} accordingly as the absolute value of its multiplier is less than, equal to, or greater than $1$. Periodic points of multiplier $0$ are called \emph{super attracting}.

A subsystem of a dynamical system is {\it minimal} if the orbit of any point in the subspace is dense in the subspace.

Now we  recall the conditions under which a number in $\Q_p$ has a square root in $\Q_p$.

An integer $a\in \mathbb{Z}$ is called a\emph{ quadratic residue modulo} $p$ if the equation $x^{2}\equiv a ~(\!\!\!\!\mod p)$ has a solution
$x\in \mathbb{Z}$. The following lemma characterizes those $p$-adic integers which admit a square root in $\Qp$.
\begin{lemma}[\cite{Mahler81}]\label{solution} Let $a$ be a nonzero $p$-adic number with its p-adic expansion
$$ a=p^{v_{p}(a)}(a_0+a_{1}p+a_{2}p^{2}+\cdots)$$
where $1\le a_0\le p-1$ and $0\leq a_j \leq p-1 \ (j\geq 1)$. The equation
$x^2=a$ has a solution $x\in \Qp$ if and only if
 the following conditions are satisfied
\, \indent \ \
  \item[{\rm (i)}] $v_{p}(a)$ is even;
 \, \indent \ \
  \item[{\rm (ii)}] $a_0$ is quadratic residue modulo $p$ if $p\neq 2$; or $a_1=a_2=0$ if $p=2$.
\end{lemma}

%A fixed point $x_0$ is called an attractor if there exists a neighborhood $V(x_0)$ of $x_0$ such that $$\forall \ y \in V (x_0), \quad\lim \phi^n(y) = x_0.$$ 
%If x0 is an attractor then its basin of attraction is
%n??
%A ( x 0 ) = { y ? Cl p : y n ? x 0 , n ? ? } .
%A fixed point x0 is called repeller if there exists a neighborhood V (x0 ) of x0 such that |f (x) ? x0 |p > |x ? x0|p for x ? V (x0), x?= x0. Let x0 be a fixed point of a function f(x). The ball Vr(x0) (contained in U) is said to be a Siegel disk if each sphere S?(x0), ? < r is an invariant sphere of f(x), i.e. if x ? S?(x0) then all iterated points xn ? S?(x0) for all n = 1, 2 . . . . The union of all Siegel desks with the center at x0 is said to a maximum Siegel disk and is denoted by SI(x0).

\medskip
\section{Dynamical structures}

Now, let $p\geq 3$. We will study the dynamical structure of the rational maps $\phi(x)=ax+\frac{1}{x}$ with $a\in \Q_p\setminus \{0\}$ on the projective line $\P$.  
In general, we can also consider the rational maps 
\[
ax+{b \over x}, \quad \text{with} \ a, b\in \Q_p.
\]
Remark that if $\sqrt{b}$ exists in $\Q_p$, then $ax+{b \over x}$ is conjugate to $ax+{1 \over x}$ through the conjugacy $x\mapsto {1 \over \sqrt{b}}x$. The case that $\sqrt{b}$ does not exist in $\Q_p$ which is not included in the present paper could be a subject for future study.

%An easy computation shows that 
%$$\phi(x)-\phi(y)=(a-\frac{1}{xy})(x-y)$$
%and $$\phi^{\prime}(x)=a-\frac{1}{x^2}.$$
The dynamical system $(\P, \phi)$ exhibits different dynamical structures according to different absolute values of $a$. When $|a|_p=1$, the transformation $\phi$ has good reduction. The dynamical structure of $(\P, \phi)$, as shown in Theorem \ref{thm2}, can be deduced directly from Theorem 1.2 of \cite{FFLW2016}. 

We are thus concerned only with the cases:
$|a|_p>1$ and $|a|_p<1$. We remark that for both cases, $\infty$ is a fixed point of $\phi$.

\subsection{Case $|a|_p>1$}
\begin{proposition}\label{prop-1}
Suppose $|a|_p>1$. If $\sqrt{1-a}\notin\Q_p$, then \[\forall x\in \Q_p, \quad \lim_{n\to\infty} \phi^n(x)=\infty.\]
\end{proposition}
\begin{proof}
By the assumption $|a|_p>1$, for all $x\in \Q_p$ such that $|x|_p\geq 1$, we have 
\[|\phi(x)|_p=\left|ax+{1\over x}\right|_p=|ax|_p>|x|_p.\]
Thus the absolute values of the iterations $\phi^n(x)$ are strictly increasing. Hence
\begin{align}\label{attr-basin}
\lim_{n\to\infty} \phi^n(x)=\infty, \quad \text{for all } x\in \Q_p, |x|_p\geq 1.
\end{align}
That is to say, $\{x\in \Q_p: |x|_p\geq 1\}$ is included in the attracting basin of $\infty$.

Now we investigate the points in the open disk $\{x\in \Q_p: |x|_p<1\}$. We partition this disk into two: $$A_1:=\left\{x\in \Q_p :|x|_p<1, |ax|_p\neq {1\over  |x|_p}\right\}, \ A_2:=\left\{x\in \Q_p : |x|_p<1, |ax|_p={1\over  |x|_p}\right\}.$$  

If $x\in A_1$, then 
\[
|\phi(x)|_p=\max\left\{|ax|_p, {1 \over |x|_p}\right\} >1.
\]
Thus by (\ref{attr-basin}), $\phi(x)$ falls into the attracting basin of $\infty$, and $\lim\limits_{n\to\infty} \phi^n(x)=\infty$.

If $|ax|_p={1\over  |x|_p}$, then $|a|_p={1\over |x|_p^2}$, which means that $v_p(a)$ is an even number. Since $|a|_p>1$, the condition $\sqrt{1-a}\not\in \Q_p$ is equivalent to that $\sqrt{-a}$ does not exist in $\Q_p$. Hence the equation $\phi(x)=0$ has no solution in $\Q_p$. Since $p\geq 3$, this is also equivalent to that the first digits of $ax$ and ${1\over x}$ in their $p$-adic expansions can not be canceled. Thus $|\phi(x)|_p=|ax|_p$ or $|{1/ x}|_p$ and hence strictly larger than $1$. Therefore, by (\ref{attr-basin}), $\lim\limits_{n\to\infty} \phi^n(x)=\infty$.

\end{proof}

\begin{lemma}Suppose $|a|_p>1$. If $\sqrt{1-a}\in\Q_p$, then $\phi$ has two repelling fixed points $$x_{1,2}=\pm\frac{1}{\sqrt{1-a}}.$$
\end{lemma}

\begin{proof}
It is easy to check that $x_{1,2}=\pm\frac{1}{\sqrt{1-a}}$ are the two fixed points of $\phi$.
Note that  
$$\phi^{\prime}(x_{1})=\phi^{\prime}(x_{2})=2a-1.$$ 
Since $|a|_p>1$,  we have 
$$|\phi^{\prime}(x_{1})|_p=|\phi^{\prime}(x_{2})|_p=|2a-1|_p>1.$$
\end{proof}

\begin{lemma}\label{scaling}
Suppose $|a|_p>1$. If $\sqrt{1-a}\in\Q_p$,  then
$$\phi(D(x_1,p^{\frac{v_p(a)}{2}-1}))=\phi(D(x_2,p^{\frac{v_p(a)}{2}-1}))=D(0, p^{-\frac{v_p(a)}{2}-1}).$$
\end{lemma}
\begin{proof}
%Since   $|a|_p>1$ and $\sqrt{1-a}\in\Q_p$,  by  Lemma \ref{solution}, $v_p(a)<0$ is even.
Since $|a|_p>1$ and $\sqrt{1-a}\in\Q_p$, the valuation $v_p(a)$ is a negative even number. Note that 
 $$|\phi(x_1)|_p=|\phi(x_2)|_p=|x_1|_p=|x_2|_p=p^{\frac{v_p(a)}{2}}.$$
% If $x,y \in D(x_1,p^{\frac{v_p(a)}{2}-1}) $, then $$\frac{1}{|xy|_p}=$$
We need only to show that for all $x,y$ in the same disk $D(x_1,p^{\frac{v_p(a)}{2}-1})$ or $D(x_2,p^{\frac{v_p(a)}{2}-1})$,
$$|\phi(x)-\phi(y)|_p=p^{-{v_p(a)}}|x-y|_p.$$
Without loss of generality, we assume that    $x,y \in D(x_1,p^{\frac{v_p(a)}{2}-1})$.
By the definition of the spherical metric on $\P$, $$ D(x_1,p^{\frac{v_p(a)}{2}-1})=\left\{\frac{1}{x}: |x-\sqrt{1-a}|_p \leq p^{-\frac{v_p(a)}{2}-1}\right\}.$$
Hence, there are $x^{\prime},y^{\prime}\in D(0,p^{-\frac{v_p(a)}{2}-1})$  such that 
$$x=\frac{1}{\sqrt{1-a}+x^\prime} \ \hbox{ and } \ y=\frac{1}{\sqrt{1-a}+y^\prime}.$$
So  we have 
$$\left|a-\frac{1}{xy}\right|_p=|2a-1-(x^{\prime}+y^{\prime})\sqrt{1-a}-x^{\prime}y^{\prime}|_p.$$
Observing that $|a|_p\geq 1$ and  $|(x^{\prime}+y^{\prime})\sqrt{1-a}|_p \leq |a|_p/p$, we have  
$$\left|a-\frac{1}{xy}\right|_p=|a|_p.$$
Hence, 
$$|\phi(x)-\phi(y)|_p=\left|(a-\frac{1}{xy})(x-y)\right|_p=|a|_p|x-y|_p=p^{-{v_p(a)}}|x-y|_p.$$

\end{proof}

\begin{lemma}\label{goinfty}Suppose $|a|_p>1$. If $\sqrt{1-a}\in\Q_p$, then for all $$x\notin D(x_1,p^{\frac{v_p(a)}{2}-1})\cup D(x_2,p^{\frac{v_p(a)}{2}-1}),$$
we have
 \[ \lim_{n\to\infty} \phi^n(x)=\infty.\]
\end{lemma}
\begin{proof}
Note that for all $x\notin D(0,1/p)$, $|1/x|_p\leq1$. Thus 
$$|\phi(x)|_p=|a|_p|x|_p>1.$$
Hence $$|\phi^n(x)|_p=|a|^n_p|x|_p,$$ 
which implies 
\[\forall \ x\notin D(0,1), \quad \lim_{n\to\infty} \phi^n(x)=\infty.\]

To finish the proof, we will show  that for all
\[ 0\neq x\in D(0,1)\setminus (D(x_1,p^{\frac{v_p(a)}{2}-1})\cup D(x_2,p^{\frac{v_p(a)}{2}-1})), \]
we have $|\phi(x)|_p>1$.

We distinguish three cases.\\
Case (1) $|x|_p<1/\sqrt{|a|_p}$. Since  $|1/x|_p > |ax|_p$, we have  
$$|\phi(x)|_p=|1/x|_p>1.$$
Case (2) $|x|_p=1/\sqrt{|a|_p}$. Observe that  $|ax|_p=|1/x|_p=\sqrt{|a|_p}$.  The assumption $ x\notin D(x_1,p^{\frac{v_p(a)}{2}-1})\cup D(x_2,p^{\frac{v_p(a)}{2}-1})$ implies  that  the  sum of the first digits of the $p$-adic expansion of $ax$ and $1/x$ is not $0$ modulo $p$, which leads to 
$$|\phi(x)|_p=|ax+1/x|_p=|ax|_p=\sqrt{|a|_p}>1.$$
\\
Case (3) $1/\sqrt{|a|_p}<|x|_p<1$. Since  $ |ax|_p>|1/x|_p>1$, we have 
$$|\phi(x)|_p= |ax|_p>1.$$
 \end{proof}

Let $(\Sigma_{2}, \sigma)$ be the full shift of two symbols.
\begin{proposition}\label{prop-2}
Suppose $|a|_p>1$ and  $\sqrt{1-a}\in\Q_p$. Then there exists an invariant set $\mathcal{J}$ such that 
$(\mathcal{J},\phi)$ is topologically conjugate to $( \Sigma_{2},\sigma),$
and $$\lim_{n\to \infty}\phi^n(x)=\infty,\quad \forall x\in \Qp\setminus \mathcal{J}.$$
\end{proposition}
\begin{proof}
By the proof  of  Lemma \ref{scaling}, we obtain that both of the restricted maps $$\phi: D(x_i,p^{\frac{v_p(a)}{2}-1})\to D(0,p^{-\frac{v_p(a)}{2}-1}), \quad i=1,2$$
are expanding and bijective. Note that  $D(x_i,p^{\frac{v_p(a)}{2}-1})\subset D(0,p^{-\frac{v_p(a)}{2}-1})$ for $i=1,2$.
Let $$\Omega=D(x_1,p^{\frac{v_p(a)}{2}-1})\cup  D(x_2,p^{\frac{v_p(a)}{2}-1})$$ and $$\mathcal{J}=\bigcap_{i=0 }^{\infty}\phi^{-n}(\Omega).$$ By Theorem 1.1 of \cite{FLWZ07}, $\mathcal{J}$ is $\phi$-invariant and
$(\mathcal{J},\phi)$ is topologically conjugate to $( \Sigma_{2},\sigma)$.
 Note that all $x\in \Omega\setminus J$ will eventually fall into $\Q_p\setminus \Omega$ by iteration of $\phi$. Thus by Lemma \ref{goinfty},   we immediately get
  $$\lim_{n\to \infty}\phi^n(x)=\infty,\  \ \forall x\in \Qp\setminus \mathcal{J}.$$
\end{proof}

\begin{proof}[Proof of Theorem \ref{main-thm-1}] It follows directly from Propositions \ref{prop-1} and \ref{prop-2}.
\end{proof}

\subsection{Case $|a|_p<1$}

We distinguish two sub cases: $\sqrt{-a}\notin \Q_p$ and $\sqrt{-a}\in \Q_p$.
\subsubsection{$\sqrt{-a}\notin \Q_p$}
\begin{lemma}\label{6}
If$\sqrt{-a}\notin \Q_p$, then 
for all $ - \lfloor  v_{p}(a)/2\rfloor \leq i \leq \lfloor  v_{p}(a)/2\rfloor$,  $$\phi(S(0,p^i))\subset   S(0,p^{-i})$$ and $\phi^2$ is $1$-Lipschitz continuous on   $S(0,p^i)\cup S(0,p^{-i})$.
\end{lemma}
\begin{proof}
If $x\in S(0,p^{i})$ for some $-\lfloor  v_{p}(a)/2\rfloor\leq  i \leq 0 $, then
by the assumption $|a|_p<1$, we have 
$$|\phi(x)|_p=|ax+1/x|_p=p^{-i}.$$
Now let $x\in S(0,p^{i})$ for some $0\leq  i \leq  \lfloor  v_{p}(a)/2\rfloor$.
When $i < \lfloor  v_{p}(a)/2\rfloor$,  we have
$$|\phi(x)|_p=|ax+1/x|_p=p^{-i}.$$
When  $v_{p}(a)/2$ is even and $i=v_{p}(a)/2$, the condition  $\sqrt{-a}\notin \Q_p$ implies that 
$$|\phi(x)|_p=|ax+1/x|_p=p^{-i}.$$
 Hence, the first assertion of the lemma holds.
 
Let us show that  $\phi^2$ is $1$-Lipschitz continuous on   $S(0,p^i)\cup S(0,p^{-i})$ for $ - \lfloor  v_{p}(a)/2\rfloor \leq i \leq \lfloor  v_{p}(a)/2\rfloor$. 
Let $x,y \in S(0,p^{i})\cap  S(0,p^{-i})$.  If $x\in S(0,p^{i})$ and  $y \in S(0,p^{-i})$, then $|xy|_p=1$ and  thus
$$|\phi(x)-\phi(y)|_p=\left|a-\frac{1}{xy}\right|_p|x-y|_p=|x-y|_p.$$
Hence, it suffices to show that  for each $ - \lfloor  v_{p}(a)/2\rfloor \leq i \leq \lfloor  v_{p}(a)/2\rfloor$, $$ \forall x,y \in S(0,p^{i}), \quad |\phi^2(x)-\phi^{2}(y)|_p\leq |x-y|_p.$$
By observing that  $|a|_p\leq 1/|xy|_p$ and $|a|_p\leq 1/|\phi(x)\phi(y)|_p$, we have 
\begin{align*}  
|\phi^2(x)-\phi^{2}(y)|_p&=  \left|a-\frac{1}{\phi(x)\phi(y)}\right|_p|\phi(x)-\phi(y)|_p\\
&=\left|a-\frac{1}{\phi(x)\phi(y)}\right|_p \left|a-\frac{1}{xy}\right|_p|x-y|_p\\
& \leq \left|\frac{1}{\phi(x)\phi(y)}\right|_p \left|\frac{1}{xy}\right|_p|x-y|_p=|x-y|_p.
\end{align*}

\end{proof}

The field $\mathbb{C}_p$ of $p$-adic complex  numbers is the metric completion of the algebraic closure $\overline{\Q}_p$ of $\Q_p$. We denote by $B(a,r):=\{x\in \mathbb{C}_p, |x|_p\leq r\}$ the closed ball in $\mathbb{C}_p$ centered at $a$ with radius $r>0$.

\begin{lemma}\label{a}
Assume $\sqrt{-a}\notin \Q_p$.  Let $x_0\in \Q_p$ with $ p^{-\lfloor  v_{p}(a)/2 \rfloor}\leq  |x_0|_p \leq p^{\lfloor  v_{p}(a)/2 \rfloor}$. Then  $\phi^2(B(x_0,|x_0|_p/p))\subset B(\phi^2(x_0),|x_0|_p/p)$.
\end{lemma}
\begin{proof}For any  $x \in B(x_0,|x_0|_p/p)$, the condition  $\sqrt{-a}\notin \Q_p$ implies 
$$|x\phi(x)|_p=1.$$ 
Hence
\begin{align*}
|\phi^2(x)-\phi^{2}(x_0)|_p&=  \left|a-\frac{1}{\phi(x)\phi(x_0)}\right|_p|\phi(x)-\phi(x_0)|_p\\
&=\left|a-\frac{1}{\phi(x)\phi(x_0)}\right|_p \left|a-\frac{1}{xx_0}\right|_p|x-x_0|_p\\
& \leq \left|\frac{1}{\phi(x)\phi(x_0)}\right|_p \left|\frac{1}{xx_0}\right|_p|x-x_0|_p=|x-x_0|_p.
\end{align*}
\end{proof}

\begin{lemma}\label{7}
Assume $\sqrt{-a}\notin \Q_p$.  Let $x_0\in \Q_p$ with $ p^{-\lfloor  v_{p}(a)/2 \rfloor}\leq  |x_0|_p \leq p^{\lfloor  v_{p}(a)/2 \rfloor}$. Then the Taylor expansion of the map  $$\phi^2:D(x_0,|x_0|_p/p)\to D(\phi^2(x_0),|x_0|_p/p) $$ can be written as  
$$\phi^2(x+x_0)=\phi^{2}(x_0)+\sum_{i=1}^{\infty}\alpha_i (x+x_0)^{i}, $$ 
with $ \alpha_i \in \Q_p$ such that $$|\alpha_i|_p\leq (p/|x_0|_p)^{i-1}.$$
\end{lemma}
\begin{proof} 
The assumption  $\sqrt{-a}\notin \Q_p$ implies that there is no $\phi$-preimage of $0$ in $\Q_p$.  However, there are two 
preimages $\pm \frac{1}{\sqrt{-a}}$ of $0$ in the quadratic extension $\Q_p(\sqrt{-a})$.  
For each  $x_0\in \Q_p$ with $ p^{-\lfloor  v_{p}(a)/2 \rfloor}\leq  |x_0|_p \leq p^{\lfloor  v_{p}(a)/2 \rfloor}$,  the condition  $\sqrt{-a}\notin \Q_p$ implies that 
$$\left|x_0-\frac{1}{\sqrt{-a}}\right|_p=\left|x_0+\frac{1}{\sqrt{-a}}\right|_p> {|x_0|_p \over p}.$$
By Lemma \ref{6}, we also  have $$\left|\phi(x_0)-\frac{1}{\sqrt{-a}}\right|_p=\left|\phi(x_0)+\frac{1}{\sqrt{-a}}\right|_p> {|\phi(x_0)|_p \over p}.$$
Hence, the ball  $$B(x_0,|x_0|_p/p):=\left\{x\in \mathbb{C}_p, |x|_p\leq \frac{|x_0|_p}{p}\right\}$$ is disjoint from the set  $\{\pm\frac{1}{\sqrt{-a}}, 0, \infty\}$ of polars of $\phi^2$. 
This implies  that the Taylor expansion of $\phi^2$ at $x_0$ 
$$\phi^2(x+x_0)=\phi^2(x_0)+\sum_{i=1}^{\infty}\alpha_i (x+x_0)^{i},$$
is convergent on the ball  $B(x_0,|x_0|_p/p)$.
By Lemma \ref{a}, we have    $\phi^2(B(x_0,|x_0|_p/p))\subset B(\phi^2(x_0),|x_0|_p/p)$.
The Newton polygon \cite[p. 249]{SilvermanGTM241} gives
$$|\alpha_i|_p\left(\frac{|x_0|_p}{p}\right)^i\leq \frac{|x_0|_p}{p}.$$
That is $|\alpha_i|_p\leq (p/|x_0|_p)^{i-1}.$
\end{proof}

\begin{lemma}\label{8}
If $\sqrt{-a}\notin \Q_p$, then 
for each $x\in \Q_{p}\setminus \{0,\infty\}$, there exists a  positive integer $N$ such that 
$$\phi^{n}(x)\in S(0,p^i)\cup S(0,p^{-i}), \quad \forall n\geq N$$   for some $0\leq i\leq \lfloor v_p(a)/2\rfloor$.
\end{lemma}
\begin{proof} 
Note that $|\phi(x)|_p=|{1}/{x}|_p$, if $|x|_p<  p^{-\lfloor v_p(a)/2\rfloor}$. Thus we have
$$\phi(D(0,p^{-\lfloor v_p(a)/2\rfloor-1}))=\P\setminus D(0,p^{ \lfloor v_p(a)/2\rfloor}).$$
It suffices to show that the statement holds for $x\in  \P\setminus D(0,p^{ \lfloor v_p(a)/2\rfloor})$.
 In fact, if $x\in  \P\setminus D(0,p^{ \lfloor v_p(a)/2\rfloor}+1) $,  one can check that 
 $$|\phi(x)|_p=|a|_p|x|_p.$$
 So there exists an integer $N$ such that $p^{- \lfloor v_p(a)/2\rfloor}\leq |\phi^N(x)|_p\leq p^{ \lfloor v_p(a)/2\rfloor}$. 
If $x\in S(0,p^{ \lfloor v_p(a)/2\rfloor})$, the conclusion is followed by Lemma \ref{6}.
  
\end{proof}

\begin{proposition}\label{prop-3}
Assume that  $\sqrt{-a}\notin \Q_p$. Then $\Q_p\setminus\{0\}$ is $\phi$-invariant
and the sub dynamical system $(\Q_p\setminus \{0\},\phi)$ can be decomposed into 
$$\P= \mathcal{P} \bigsqcup \mathcal{M} \bigsqcup \mathcal{B}
$$
where $ \mathcal{P}$ is the finite set consisting of all periodic points of
$\phi$, $\mathcal{M}= \bigsqcup_i \mathcal{M}_i$ is the union of all (at most countably
many) clopen invariant sets such that each $\mathcal{M}_i$ is a finite union
of balls and each subsystem $\phi: \mathcal{M}_i \to \mathcal{M}_i$ is minimal, and each
point in $\mathcal{B}$ lies in the attracting basin of a periodic orbit or of
a minimal subsystem.
\end{proposition}
\begin{proof}
By Lemmas \ref{6} and \ref{8}, it suffices to show that the sub dynamical system $$(S(0,p^i), \phi^2)$$ has a minimal decomposition as stated in the proposition for all  $ -\lfloor v_p(a)/2\rfloor \leq i \leq  \lfloor v_p(a)/2\rfloor.$ 
Note  that $\phi^2$ is equicontinuous on  $S(0,p^i)$ and $$S(0,p^i)=\bigcup_{j=1}^{p-1} D(jp^{-i}, p^{i-1}).$$
Thus for each disk $D(jp^{-i}, p^{i-1})$ in $S(0,p^i)$, its image $\phi^2(D(jp^{-i}, p^{i-1})$ is still a disk in $S(0,p^i)$.
Hence $\phi^2$ induces  a map $\overline{\phi^2}$ from the set $\{D(jp^{-i}, p^{i-1}): j=\{1,2,\cdots,p-1\} \}$ of disks to itself:
\[
\overline{\phi^2}(D(jp^{-i}, p^{i-1})):=\phi^2(D(jp^{-i}, p^{i-1})).
\]

Assume that $(D(x_1p^{-i}, p^{i-1}),D(x_2p^{-i}, p^{i-1}), \cdots, D(x_kp^{-i}, p^{i-1}))$ is a $k$-cycle of $\overline{\phi^2}$, i.e.
\begin{align*}
\overline{\phi^2}(D(x_1p^{-i}, p^{i-1}))=D(x_2p^{-i}, p^{i-1})\\
\overline{\phi^2}(D(x_2p^{-i}, p^{i-1}))=D(x_3p^{-i}, p^{i-1}),\\
 \cdots  \quad \quad \quad \quad \quad \quad  \quad \quad \quad \\
 \overline{\phi^2}(D(x_kp^{-i}, p^{i-1}))=D(x_1p^{-i}, p^{i-1}).
\end{align*}
Then $\phi^{2k}$ is a transformation on $D(x_1p^{-i}, p^{i-1})$.
We thus study the dynamical system  $(D(x_1p^{-i}, p^{i-1}), \phi^{2k})$. By Lemma \ref{8}, the Taylor expansion of the map  $$\phi^{2k}:D(x_1p^{-i}, p^{i-1})\to D(x_1p^{-i}, p^{i-1}) $$ can be written as  
$$\phi^{2k}(x+x_1p^{-i})=\phi^{2k}(x_1p^{-i})+\sum_{j=1}^{\infty}\alpha_j (x+x_1p^{-i})^{j}, $$ 
with $ \alpha_j \in \Q_p$ such that 
\begin{align}\label{aaa}
|\alpha_j|_p\leq 1/p^{(i-1)(j-1)}.
\end{align}
Thus one can check that the dynamical system $(D(x_1p^{-i}, p^{i-1}), \phi^{2k})$ is conjugate to a dynamical system on $\mathbb{Z}_p$, denoted by $(\mathbb{Z}_p,\psi)$, through the conjugacy
$$f:D(x_1p^{-i}, p^{i-1})\to D(0,1), \hbox{ with } f(x)=p^{i-1}(x-x_1p^{-i}).$$ 
In fact, the inequality (\ref{aaa}) implies that $\psi$ is a convergent  series with integer coefficients. By Theorem 1.1 of \cite{FLpre}, 
the system $(\mathbb{Z}_p,\psi)$ has a minimal decomposition.
Hence, the conjugated dynamical system $(D(x_1p^{-i}, p^{i-1}), \phi^{2k})$ has a corresponding minimal decomposition which implies  that 
 $$(S(0,p^i), \phi^2)$$ has a minimal decomposition as stated in the proposition for all  $-\lfloor v_p(a)/2\rfloor \leq i \leq  \lfloor v_p(a)/2\rfloor.$

\end{proof}

%\begin{lemma}
%For $0\leq i \leq \lfloor  v_{p}(a)/2\rfloor-1$, then $S(0,p^i)\cup S(0,p^{-i})$ is $\phi$-invariant  and $\phi^2$ is $1$-Lipschitz continuous on   $S(0,p^i)\cup S(0,p^{-i})$.
%\end{lemma}
%
%\begin{lemma}
%If $\sqrt{-a}\notin \Q_p$, then $S(0,p^{v_{p}(a)/2})\cup S(0,p^{-v_{p}(a)/2})$ is $\phi$-invariant  and $\phi^2$ is $1$-Lipschitz continuous on $S(0,p^{v_{p}(a)/2})\cup S(0,p^{-v_{p}(a)/2})$.
%\end{lemma}

%
%\begin{lemma}
%For each $x\in \Q_{p}\setminus \{\infty\}$, there exists a  positive integer $n$ such that 
%$$\phi^{n}(x)\in S(0,p^i)\cup S(0,p^{-i})$$   for some $0\leq i\leq \lfloor v_p(a)/2\rfloor$.
%\end{lemma}

\subsubsection{Case $\sqrt{-a}\in \Q_p$}

\begin{lemma}\label{gotozero}
If $|a|<1$ and $\sqrt{-a}\in \Q_p$, then $$\phi\left(D\Big( \pm {1 \over \sqrt{-a}}, \frac{1}{p\sqrt{|a|_p}}\Big)\right)= D\Big(0, \frac{\sqrt{|a|_p}}{p}\Big),$$  and for all $x,y \in D( \pm {1 \over \sqrt{-a}}, \frac{1}{p\sqrt{|a|_p}})$,
$$|\phi(x)-\phi(y)|_p= |a|_p |x-y|_p.$$.
\end{lemma}
\begin{proof}
Note that  $ \phi(\pm 1/\sqrt{-a})=0$. It suffices to show that 
$$\forall x, y \in D( \pm {1 \over \sqrt{-a}}, \frac{1}{p\sqrt{|a|_p}}),\quad |\phi(x)-\phi(y)|_p= |a|_p |x-y|_p.$$ 
Without loss of generality, we assume that $x,y \in D( {1 \over \sqrt{-a}}, \frac{1}{p\sqrt{|a|_p}})$. By the same arguments in the proof of Lemma \ref{scaling}, there exist $x^{\prime},y^{\prime}\in D(0, \frac{\sqrt{|a|_p}}{p})$
such that 
$$x=\frac{1}{\sqrt{-a}+x^{\prime}},  \quad y=\frac{1}{\sqrt{-a}+y^{\prime}},$$
and
$$\left|a-\frac{1}{xy}\right|_p= \left|2a-(x^{\prime}+y^{\prime})\sqrt{-a}-x^{\prime}y^{\prime}\right|_p.$$
Since $x^{\prime},y^{\prime}\in D(0, \frac{\sqrt{|a|_p}}{p})$, we immediately get 
 $$\left|a-\frac{1}{xy}\right|_p=|a|_p.$$
So 
$$|\phi(x)-\phi(y)|_p=\left|a-\frac{1}{xy}\right|_p |x-y|_p=|a|_p |x-y|_p.$$
\end{proof}

\begin{proposition}\label{prop-4}
If $|a|_p<1$ and $\sqrt{-a}\in \Q_p$, then there exists an $\phi$-invariant subset $\mathcal{J}$ such that 
$(\mathcal{J}, \phi)$ is topologically conjugate to a subshift of finite type $(\Sigma_{A},\sigma),$
where the transition matrix is
$$A=\begin{pmatrix} 0 & 0& 1 & 0   \\ 0&0&1&0\\  0&0&0&1 \\  1&1&0&1 \end{pmatrix}.$$
The topological entropy of the system $(\P, \phi)$ is larger than $\log 1.69562...$.
 \end{proposition}
 
\begin{proof}
Let $D_1=D( 1/\sqrt{-a}, 1/p),$ $ D_2= D(- 1/\sqrt{-a}, 1/p)$, 
$D_3=D( 0, {|a|_p}/{p})$   and $D_4=\P\setminus D(0, |a|_p)$. 
By Lemma \ref{gotozero}, the restricted maps $\phi:D_1 \to D_3$ and  $\phi:D_2 \to D_3$  are both bijective. 
One can also check directly that  $\phi:D_3\to D_4$ and  $\phi: D_4\to \phi(D_4)=\P\setminus D(0,1)$ are  
bijective. 

Set  $\Omega=\bigcup_{i=1}^{4}D_i$ and consider the restricted map $\phi: \Omega\to \P$. Let $f: x\mapsto1/(x-1)$ and $\psi:= f\circ\phi\circ f^{-1}$. 
We  study the map $\psi:  f(\Omega)\to \P$. Observe that $f(\Omega)\subset \mathbb{Z}_p$. One can check that $\psi$ satisfies the conditions in  \cite{FLWZ07}.  Thus by the main result of \cite{FLWZ07},   $$\mathcal{J^{\prime}}=\bigcap_{i=0}\psi^{-i}(\Omega)$$ is an invariant set of $\psi$ and $(\mathcal{J^{\prime}}, \psi)$ is topologically conjugate to the subshift of finite type $(\Sigma_{A},\sigma)$ with $$A=\begin{pmatrix} 0 & 0& 1 & 0   \\ 0&0&1&0\\  0&0&0&1 \\  1&1&0&1 \end{pmatrix}.$$
Since $\psi= f\circ\phi\circ f^{-1}$, we deduce that
$(\mathcal{J}, \phi)$ is topologically conjugate to $(\Sigma_{A},\sigma).$

The topological entropy of $(\Sigma_{A},\sigma)$ is $\log (1.69562...)$ where $1.69562...$ is the maximal eigenvalue of the matrix $A$. Since $(\Sigma_{A},\sigma)$ is topologically conjugate to a subsystem of $(\P, \phi)$, we confirm the last assertion of the proposition.
\end{proof}
%Remark that  $(\Sigma_{A},\sigma)$ has positive entropy, which implies that  $(\P,\phi)$ also has positive entropy. 
We remark that even though a chaotic subsystem is well described, the detailed dynamical structure of $\phi$ on the whole space for the case $|a|_p<1$ and $\sqrt{-a}\in \Q_p$ is far from clear. There may exist more complicated sub dynamical systems.

\begin{proof}[Proof of Theorem \ref{main-thm-2}] It follows directly from Propositions \ref{prop-3} and \ref{prop-4}.
\end{proof}

\bibliographystyle{plain}
%\bibliography{ref}

\end{document}